\documentclass{amsart}
\usepackage[utf8]{inputenc}
\usepackage{amssymb}
\usepackage{hyperref}
\usepackage[final]{showkeys} 

\input xy
\xyoption{all}

\theoremstyle{definition}
\newtheorem{mydef}{Definition}[section]
\newtheorem{lem}[mydef]{Lemma}
\newtheorem{thm}[mydef]{Theorem}

\newtheorem{cor}[mydef]{Corollary}

\newtheorem{defin}[mydef]{Definition}

\newtheorem{remark}[mydef]{Remark}

\newtheorem{fact}[mydef]{Fact}

\newcommand{\fct}[2]{{}^{#1}#2}

\newcommand{\sea}{\mathfrak{C}}

\newcommand{\seq}[1]{\langle #1 \rangle}
\newcommand{\rest}{\upharpoonright}

\newcommand{\leap}[1]{\le_{#1}}

\newcommand{\lea}{\leap{\K}}

\newcommand{\K}{\mathbf{K}}

\newbox\noforkbox \newdimen\forklinewidth
\forklinewidth=0.3pt \setbox0\hbox{$\textstyle\smile$}
\setbox1\hbox to \wd0{\hfil\vrule width \forklinewidth depth-2pt
 height 10pt \hfil}
\wd1=0 cm \setbox\noforkbox\hbox{\lower 2pt\box1\lower
2pt\box0\relax}
\def\unionstick{\mathop{\copy\noforkbox}\limits}

\def\1nf{\unionstick^{(1)}}

\def\2nf{\unionstick^{(2)}}
\def\3nf{\unionstick^{(3)}}

\newcommand{\gtp}{\text{gtp}}

\newcommand{\gS}{\text{gS}}

\newcommand{\ehanf}[1]{\beth_{\left(2^{#1}\right)^+}}

\newcommand{\LS}{\text{LS}}

\title{On the uniqueness property of forking in abstract elementary classes}
\date{\today
}

\author{Sebastien Vasey}
\address{Department of Mathematical Sciences \\ Carnegie Mellon University \\ Pittsburgh, Pennsylvania, USA}
\email{sebv@cmu.edu}
\urladdr{http://math.cmu.edu/\textasciitilde svasey/}

\begin{document}

\begin{abstract}
  In the setup of abstract elementary classes satisfying a local version of superstability, we prove the uniqueness property for $\mu$-forking, a certain independence notion arising from splitting. This had been a longstanding technical difficulty when constructing forking-like notions in this setup. As an application, we show that the two versions of forking symmetry appearing in the literature (the one defined by Shelah for good frames and the one defined by VanDieren for splitting) are equivalent. 
\end{abstract}

\maketitle

\tableofcontents

\section{Introduction}

In the study of classification theory for abstract elementary classes (AECs), the question of when a forking-like notion exists is central. The present paper is a contribution to this problem.

To state our result more precisely, we first recall that there is a semantic notion of type in AECs: for the rest of this introduction we fix an AEC $\K$ with amalgamation, joint embedding, and arbitrarily large models. This allows us to fix a big universal model-homogeneous\footnote{$M$ is \emph{model-homogeneous} if whenever $M_0 \lea N_0$ are such that $M_0 \lea M$ and $\|N_0\| < \|M\|$, then $N_0$ embeds inside $M$ over $M_0$.} monster model $\sea$ and work inside it. For $M \lea \sea$ and $a \in \sea$, let $\gtp (a / M)$ (the Galois, or orbital, type of $a$ over $M$) be the orbit of $a$ under the automorphisms of $\sea$ fixing $M$ (Galois types can be defined without any assumptions on $\K$, but then the definition becomes more technical). Write $\gS (M)$ for the set of all Galois types over $M$. The definitions of stability and saturation are as expected. Two important results of Shelah are:

\begin{enumerate}
\item \cite[II.1.14]{shelahaecbook} If $M$ is saturated, then $M$ is model-homogeneous.
  \item \cite[II.1.16]{shelahaecbook} If $\K$ is stable in $\mu$ and $M \in \K_\mu$, then there exists $N \in \K_\mu$ universal over $M$.
\end{enumerate}

To motivate the main result of this paper, let us first consider the following consequence:

\begin{cor}\label{easy-cor}
  Let $\K$ be an AEC with amalgamation, joint embedding, and arbitrarily large models. Let $\LS (\K) < \mu < \lambda$ be given. If $\K$ is categorical in $\lambda$, then there is a relation ``$p$ does not $\mu$-fork over $M$'' defined for $M \lea N$ both \emph{saturated models} in $\K_\mu$ and $p \in \gS (N)$ satisfying:

  \begin{enumerate}
  \item\label{cond-1} The usual invariance and monotonicity properties.
  \item\label{cond-2} Existence-extension: for $M \lea N$ both saturated in $\K_\mu$, any $p \in \gS (M)$ has a $\mu$-nonforking extension to $\gS (N)$.
  \item\label{cond-3} Uniqueness\footnote{This can also be described as ``types over saturated models are stationary''.}: for $M \lea N$ both saturated in $\K_\mu$, if $p, q \in \gS (N)$ do not $\mu$-fork over $M$ and $p \rest M = q \rest M$, then $p = q$.
  \item\label{cond-4} Symmetry: for $M$ saturated in $\K_\mu$ and $a, b \in \sea$, the following are equivalent:
    \begin{enumerate}
    \item There exists $M_a$ saturated in $\K_\mu$ containing $a$ such that $M \lea M_a$ and $\gtp (b / M_a)$ does not $\mu$-fork over $M$.
    \item There exists $M_b$ saturated in $\K_\mu$ containing $b$ such that $M \lea M_b$ and $\gtp (a / M_b)$ does not $\mu$-fork over $M$.
    \end{enumerate}
  \item\label{cond-5} Local character for universal chains: if $\delta < \mu^+$ is a limit ordinal, $\seq{M_i : i \le \delta}$ is an increasing continuous sequence of saturated models in $\K_\mu$ with $M_{i + 1}$ universal over $M_i$ for all $i < \delta$, then for any $p \in \gS (M_\delta)$ there exists $i < \delta$ such that $p$ does not $\mu$-fork over $M_i$.

  \end{enumerate}
\end{cor}

We give a proof at the end of this introduction. Several remarks are in order.

First remark: we work only over models of a fixed cardinality, so we deal with a (potentially) different nonforking relation for each cardinal $\mu$. Note in particular that the uniqueness property is for types over models of the same size, so there are no obvious relationships between $\mu_0$-forking and $\mu_1$-forking (for $\LS (\K) < \mu_0 < \mu_1 < \lambda$).

Second remark: we work only over \emph{saturated} models. We do not know how to generalize our result to \emph{all} models of cardinality $\mu$. It is worth mentioning that in the setup of Corollary \ref{easy-cor} the $\mu$-saturated models are closed under unions \cite[5.7(3)]{categ-saturated-v3}. In fact they form an AEC with Löwenheim-Skolem-Tarski number $\mu$.

Third remark: it is known (using an argument of Morley, see \cite[I.1.7(a)]{sh394}) that in the setup of Corollary \ref{easy-cor}, $\K$ is stable in $\mu$. Moreover (\ref{cond-5}) can be seen as a version of superstability: it is a replacement for ``every type does not fork over a finite set''. In fact (\ref{cond-5}) is equivalent to superstability if $\K$ is first-order axiomatizable \cite{gv-superstability-v5-toappear}.

Fourth remark: if we strengthen condition (\ref{cond-5}) to:

\begin{itemize}
  \item[(\ref{cond-5}+)] Local character: if $\delta < \mu^+$ is a limit ordinal, $\seq{M_i : i \le \delta}$ is an increasing continuous sequence of saturated models in $\K_\mu$, then for any $p \in\gS (M_\delta)$ there exists $i < \delta$ such that $p$ does not $\mu$-fork over $M_i$.
\end{itemize}

(note the difference with (\ref{cond-5}): we do \emph{not} require that $M_{i + 1}$ be universal over $M_i$) then we have arrived to Shelah's definition of a (type-full) good $\mu$-frame \cite[Definition II.2.1]{shelahaecbook}. Good frames are the main concept in Shelah's books \cite{shelahaecbook, shelahaecbook2} on classification theory for AECs. They have several applications, including the author's proof of the eventual categoricity conjecture for universal classes \cite{ap-universal-v11-toappear,categ-universal-2-selecta}. Thus the existence question for them is important. 

Fifth remark: if we add to the assumptions of Corollary \ref{easy-cor} that Galois types over saturated models of size $\mu$ are determined by their restrictions to model of size $\chi$, for some $\chi < \mu$ (this is called weak tameness in the literature), then the conclusion is known (see \cite[6.4]{vv-symmetry-transfer-afml} and \cite[5.7(1)]{categ-saturated-v3}) and one can strengthen (\ref{cond-5}) to (\ref{cond-5}+), i.e.\ one gets a good $\mu$-frame. It is known how to derive eventual weak tameness from categoricity in a high-enough cardinal, thus the conclusion also holds if $\mu$ is ``high-enough'' ($\mu \ge \ehanf{\LS (\K)}$ suffices) \cite[5.7(5)]{categ-saturated-v3}. However we are interested in arbitrary, potentially small, $\mu$. In this case the conclusion of Corollary \ref{easy-cor} is new.

Sixth remark: we actually prove a more local statement than Corollary \ref{easy-cor}: let us take a step back and explain how Corollary \ref{easy-cor} is proven. As is customary, we first study an independence notion called $\mu$-splitting \cite[3.2]{sh394}: For $M \lea N$ both in $\K_\mu$, $p \in \gS (N)$ \emph{$\mu$-splits} over $M$ if there exists $N_1, N_2 \in \K_\mu$ with $M \lea N_\ell \lea N$ for $\ell = 1,2$ and $f: N_1 \cong_M N_2$ such that $f (p \rest N_1) \neq p \rest N_2$. In the context of Corollary \ref{easy-cor}, Shelah and Villaveces (see Fact \ref{shvi}) have shown that $\mu$-splitting satisfies (\ref{cond-5}). $\mu$-splitting also satisfies weak analogs of uniqueness and extension (see Fact \ref{ext-uq}).

The weak uniqueness statement is the following: if $M_0 \lea M \lea N$ are all in $\K_\mu$, $M$ is universal over $M_0$, $p, q \in \gS (N)$ both do not $\mu$-split over $M_0$ and $p \rest M = q \rest M$, then $p = q$. Thus it is natural to define forking by ``shifting'' splitting by a universal model (this is already implicit in \cite{sh394} but is defined explicitly for the first time in \cite[3.8]{ss-tame-jsl}). Let us say that $p \in \gS (N)$ \emph{does not $\mu$-fork} over $M$ if there exists $M_0 \lea M$ such that $M$ is universal over $M_0$ and $p$ does not \emph{$\mu$-split} over $M_0$ (see Definition \ref{forking-def}; it can be shown that any reasonable forking-like notion must be $\mu$-forking over saturated models \cite[9.7]{indep-aec-apal}). In the setup of Corollary \ref{easy-cor}, it was known that $\mu$-forking satisfies all the conditions there except (\ref{cond-3}) (for symmetry, this is a recent result of the author \cite[5.7(1)]{categ-saturated-v3}, relying on joint work with VanDieren \cite{vv-symmetry-transfer-afml}). 

Let us describe the problem in proving uniqueness: let $M \lea N$ both be saturated in $\K_\mu$ and $p, q \in \gS (N)$ be not $\mu$-forking over $M$ with $p \rest M = q \rest M$. Thus we have witnesses $M_p, M_q$ such that $M$ is universal over both $M_p$ and $M_q$, $p$ does not $\mu$-split over $M_p$ and $q$ does not $\mu$-split over $M_q$. If we knew that $M_p$ and $M_q$ were the same (or at least had a common extension over which $M$ is still universal), then we could use the weak uniqueness described in the previous paragraph. However we do not know how the witnesses fit together, so we are stuck. This causes several technical difficulties, forcing for example the witnesses to be carried over in the study of towers in \cite{shvi635, vandierennomax, gvv-mlq, vandieren-symmetry-apal, vv-symmetry-transfer-afml}. In this paper, we prove the uniqueness property.

This implies for example that the equivalence relation $\approx$ defined in \cite[Definition 3.2.1]{shvi635} is just equality (Shelah and Villaveces ask if this is the case in the remarks after \cite[3.2.2]{shvi635}). Thus the machinery of strong types introduced there can essentially be dispensed with (see also the discussion in \cite[Section 3]{gvv-mlq}). The result also sheds light on \cite[Exercise 12.9]{baldwinbook09} (the ``transitivity'' of splitting). It is pointed out in \cite{baldwinbook09-errata-0316} that there is an error in this exercise, since the natural proof does not work. We are still unable to prove transitivity directly, but the result of this paper shows how to bypass it: work with $\mu$-nonforking (over limit models) instead of $\mu$-splitting. Then transitivity will follow directly from existence-extension and the uniqueness proven in this paper.

We now state our local uniqueness result more precisely. Let us say that an AEC $\K$ is \emph{$\mu$-superstable} if  $\K_\mu$ is nonempty, has amalgamation, joint embedding, no maximal models, is stable in $\mu$, and $\mu$-splitting satisfies (\ref{cond-5}) (see Definition \ref{ss-def}). The main result of this paper is:

\textbf{Theorem \ref{main-thm}.} 
  If $\K$ is $\mu$-superstable, then $\mu$-forking has the uniqueness property over limit models in $\K_\mu$.

Recall that $M$ is \emph{limit} if it is the union of an increasing continuous chain in $\K_\mu$ of the form $\seq{M_i : i \le \delta}$, $\delta < \mu^+$ limit and $M_{i + 1}$ universal over $M_i$ for all $i < \delta$. Limit models are a replacement for saturated models in a local context where we only know information about models of a single cardinality (see \cite{gvv-mlq} for an introduction to the theory of limit models). The proof of Theorem \ref{main-thm} proceeds by contradiction: if uniqueness fails, then we can build a tree of failures and this contradicts stability.

With Theorem \ref{main-thm} stated, we can now give a full proof of Corollary \ref{easy-cor}:

\begin{proof}[Proof of Corollary \ref{easy-cor}]
  By Fact \ref{shvi}, $\K$ is $\mu$-superstable. By \cite[5.7]{categ-saturated-v3}, saturated models in $\K_\mu$ are the same as limit models. Therefore Theorem \ref{main-thm} applies. We have that $\mu$-forking (from Definition \ref{forking-def}) satisfies (\ref{cond-5}). By Fact \ref{ext-uq}, it also satisfies (\ref{cond-2}) and it is clear that it satisfies (\ref{cond-1}). By Theorem \ref{main-thm}, it satisfies (\ref{cond-2}). Finally, by \cite[5.7(1)]{categ-saturated-v3} it satisfies (\ref{cond-4}).
\end{proof}

As an application of Theorem \ref{main-thm}, we can show that the symmetry property for splitting introduced by VanDieren in \cite{vandieren-symmetry-apal} (which in essence is a symmetry property for $\mu$-forking with certain uniformity requirements on the witnesses) is the same as the symmetry property given in the statement of Corollary \ref{easy-cor}: see Corollary \ref{sym-equiv-cor}. Thus the ``hierarchy of symmetry properties'' described in \cite[\S4]{vv-symmetry-transfer-afml} collapses: all the properties there are equivalent. This answers (the first part of) \cite[Question 4.13]{vv-symmetry-transfer-afml} and gives further evidence that symmetry is a natural property to study in this local context. We do not know whether symmetry follows from $\mu$-superstability. We also do not know whether in Theorem \ref{main-thm} we can assume only stability in $\mu$ (and amalgamation, etc.) rather than superstability.

Another open problem would be to study the properties of the weak kind of good frames derived in Corollary \ref{easy-cor}. They are called $H$-almost good frames by Shelah (see \cite[VII.5.9]{shelahaecbook2} and \cite{ShF841}). There has been some work on \emph{almost} (not $H$-almost) good frames (see \cite[VII.5]{shelahaecbook2}, \cite{jash940-v1}), where in addition to (\ref{cond-5}) a continuity property is required for all chains (i.e.\ given an increasing union of types where all the elements do not fork over a common model, the union of the chain does not fork over this model). In particular, conditions are given under which almost good frames are good frames. It would be interesting to know whether similar statements hold for $H$-almost good frames.

This paper was written while the author was working on a Ph.D.\ thesis under the direction of Rami Grossberg at Carnegie Mellon University and he would like to thank Professor Grossberg for his guidance and assistance in his research in general and in this work specifically. The author also thanks the referee for comments that helped improve the presentation of this paper.

\section{The main theorem}

For the rest of this paper, we assume that the reader has some basic familiarity with AECs (\cite[Chapters 4-12]{baldwinbook09} should be more than enough). We work inside a fixed AEC $\K$.

The following definition is implicit already in \cite{sh394} and is studied in several papers including \cite{shvi635, vandierennomax, gvv-mlq, vandieren-symmetry-apal, vv-symmetry-transfer-afml}. It is given the name superstability for the first time in \cite[7.12]{grossberg2002}.

\begin{defin}\label{ss-def}
  $\K$ is \emph{$\mu$-superstable} if:

  \begin{enumerate}
  \item $\mu \ge \LS (\K)$ and $\K_\mu \neq \emptyset$.
  \item $\K_\mu$ has amalgamation, joint embedding, and no maximal models.
  \item $\K$ is stable in $\mu$.
  \item $\K$ has no long $\mu$-splitting chains: for any limit ordinal $\delta < \mu^+$, any increasing continuous chain $\seq{M_i : i \le \delta}$ with $M_{i +1}$ universal over $M_i \in \K_\mu$ for all $i < \delta$, and any $p \in \gS (M_\delta)$, there exists $i < \delta$ such that $p$ does not $\mu$-split over $M_i$.
  \end{enumerate}
\end{defin}

A justification for this rather technical definition is the fact that it follows from categoricity. This is proven (with slightly different hypotheses) in \cite[2.2.1]{shvi635}. For an exposition and complete proof, see \cite{shvi-notes-apal}.

\begin{fact}\label{shvi}
  Assume that $\K$ has amalgamation and no maximal models. Let $\LS (\K) \le \mu < \lambda$. If $\K$ is categorical in $\lambda$, then $\K$ is $\mu$-superstable.
\end{fact}

\emph{From now on, we assume that $\K$ is $\mu$-superstable} (we will repeat this hypothesis at the beginning of important statements). We fix a ``monster model'' $\sea \in \K_{\mu^+}$ that is universal and model-homogeneous and work inside it.

\begin{remark}
  We could work in the more general setup of \cite{shvi635} (with only density of amalgamation bases, existence of universal extensions, limit models being amalgamation bases, and no long splitting chains), but we prefer to avoid technicalities.
\end{remark}

The following is the main object of study of this paper:

\begin{defin}[3.8 in \cite{ss-tame-jsl}]\label{forking-def}
  For $M \lea N$ both in $\K_\mu$, $p \in \gS (N)$ \emph{does not $\mu$-fork over $(M_0, M)$} if $M$ is universal over $M_0$ and $p$ does not $\mu$-split over $M_0$. We say that $p$ \emph{does not $\mu$-fork over $M$} if it does not $\mu$-fork over $(M_0, M)$ for some $M_0$.
\end{defin}

Since $\mu$ is always clear from context, we will omit it: we will say ``$p$ does not fork'' and ``$p$ does not split'' instead of ``$p$ does not $\mu$-fork'' and ``$p$ does not $\mu$-split''.

It is clear that forking has the basic invariance and monotonicity properties (see \cite[3.9]{ss-tame-jsl}). The following are implicit in \cite{sh394} and stated explicitly in \cite[I.4.10, I.4.12]{vandierennomax}. We will use them without much comments.

\begin{fact}\label{ext-uq}
  Let $M_0 \lea M \lea N \lea N'$ all be in $\K_\mu$.
  
  \begin{enumerate}
  \item Extension: If $p \in \gS (N)$ does not fork over $(M_0, M)$, then there exists an extension $q \in \gS (N')$ of $p$ that does not fork over $(M_0, M)$.
  \item Weak uniqueness: If $p, q \in \gS (N)$ do not fork over $(M_0, M)$ and $p \rest M = q \rest M$, then $p = q$.
  \end{enumerate}
\end{fact}

We now state a weak version of the conjugation property that types enjoy in good frames \cite[III.1.21]{shelahaecbook}. This will be key in the proof of the main theorem.

\begin{defin}
  Let $M, M' \in \K_\mu$, $p \in \gS (M)$, $p' \in \gS (M')$. Let $A \subseteq |M| \cap |M'|$. We say that $p$ and $p'$ are \emph{conjugate over $A$} if there exists $f: M \cong_A M'$ such that $p' = f (p)$.
\end{defin}

\begin{fact}[Conjugation property]\label{conj-prop}
  Let $\delta < \mu^+$ be a limit ordinal. Let $M_0, M, N \in \K_\mu$, with $M_0 \lea M \lea N$. Assume that $M$ is $(\mu, \delta)$-limit over $M_0$ and $N$ is $(\mu, \delta)$-limit over $M$. If $p \in \gS (N)$ does not fork over $(M_0, M)$, then $p$ and $p \rest M$ are conjugate over $M_0$.
\end{fact}
\begin{proof}
  Since $M$ is limit over $M_0$, there exists $M_1 \in \K_\mu$ such that $M_0 \lea M_1 \lea M$, $M_1$ is universal over $M_0$, and $M$ is $(\mu, \delta)$-limit over $M_1$. Note that then also $N$ is $(\mu, \delta)$-limit over $M_1$. Using uniqueness of limit models of the same length, pick $f: N \cong_{M_1} M$. Let $q := f (p)$. We claim that $q = p \rest M$. Note that by invariance $q$ does not fork over $(M_0, f[M])$, hence (by monotonicity) over $(M_0, M_1)$. By assumption and monotonicity, also $p \rest M$ does not fork over $(M_0, M_1)$. Since $f$ fixes $M_1$, $p \rest M_1 = q \rest M_1$, so using weak uniqueness $q = p \rest M$, as desired.
\end{proof}
\begin{remark}
  We do \emph{not} know here that limit models of different lengths are isomorphic.
\end{remark}

The next fact says that certain increasing chains of types have least upper bounds. The proof combines the ``no long $\mu$-splitting chains'' clause in the definition of $\mu$-superstability together with the extension and weak uniqueness properties described earlier

\begin{fact}\label{limit-fact}
  Assume that $\K$ is $\mu$-superstable. Let $\delta < \mu^+$ be a limit ordinal and let $\seq{M_i : i \le \delta}$ be increasing continuous in $\K_\mu$ with $M_{i + 1}$ universal over $M_i$ for all $i < \delta$. Suppose we are given an increasing chain of types $\seq{p_i : i < \delta}$ such that $p_i \in \gS (M_i)$ for all $i < \delta$. Then there exists a unique $p_\delta \in \gS (M_\delta)$ such that $p_\delta \rest M_i = p_i$ for all $i < \delta$.
\end{fact}
\begin{proof}
  Without loss of generality, $\delta$ is regular. If $\delta = \omega$, the conclusion is given by a straightforward direct limit argument \cite[11.1]{baldwinbook09}, so assume that $\delta > \omega$. Using no long splitting chains, for each limit $i < \delta$ there exists $j_i < i$ such that $p_i$ does not split over $M_{j_i}$. By Fodor's lemma, there exists a stationary $S \subseteq \delta$ and a $j < \delta$ such that $p_i$ does not split over $M_j$ for all $i \in S$. Since $S$ is unbounded and the $p_i$'s are increasing, $p_i$ does not split over $M_j$ for \emph{all} $i \in [j, \delta)$. Let $q \in \gS (M_\delta)$ be an extension of $p_{j + 1}$ that does not split over $M_j$. By weak uniqueness, $q \rest M_i = p_i$ for all $i \in [j + 1, \delta)$. This proves existence and uniqueness is similar: any $q' \in \gS (M_\delta)$ extending all the $p_i$'s must be nonsplitting over $M_j$, so use weak uniqueness.
\end{proof}

Recall that our goal is to prove uniqueness of nonforking extension. To this end, we define a type to be \emph{bad} if it witnesses a failure of uniqueness. We then close this definition under nonforking extensions.

\begin{defin}\label{bad-def}
  Let $M \in \K_\mu$ be limit. We define by induction on $n < \omega$ what it means for a type $p \in \gS (M)$ to be \emph{$n$-bad}:

  \begin{enumerate}
  \item $p$ is \emph{$0$-bad} if there exists a limit model $N \in \K_\mu$ with $M \lea N$ and $q_1, q_2 \in \gS (N)$ such that:

    \begin{enumerate}
    \item Both $q_1$ and $q_2$ extend $p$.
    \item $q_1 \neq q_2$.
    \item Both $q_1$ and $q_2$ do not fork over $M$.
    \end{enumerate}
  \item For $n < \omega$, $p$ is \emph{$(n + 1)$-bad} if there exists a limit model $M_0 \in \K_\mu$ with $M_0 \lea M$ such that $p \rest M_0$ is $n$-bad and $p$ does not fork over $M_0$.
  \item $p$ is \emph{bad} if $p$ is $n$-bad for some $n < \omega$.
  \end{enumerate}
\end{defin}

The following is an easy consequence of the definition (in fact the definition is tailored exactly to make this work):

\begin{remark}\label{bad-trans}
  Let $M \lea N$ both be limit in $\K_\mu$. If $p \in \gS (N)$ does not fork over $M$ and $p \rest M$ is bad, then $p$ is bad.
\end{remark}

We now proceed to develop some the theory of bad types. In the end, we will conclude that this contradicts stability in $\mu$, hence there cannot be any bad types. The next two lemmas are crucial: bad types are closed under unions of universal chains, and any bad type has two distinct bad extensions.

\begin{lem}\label{bad-limit}
  Assume that $\K$ is $\mu$-superstable. Let $\delta < \mu^+$ be a limit ordinal. Let $\seq{M_i : i \le \delta}$ be an increasing continuous chain of limit models in $\K_\mu$ with $M_{i + 1}$ limit over $M_i$ for all $i < \delta$. Let $\seq{p_i : i \le \delta}$ be an increasing chain of types, with $p_i \in \gS (M_i)$ for all $i < \delta$. If $p_i$ is bad for all $i < \delta$, then $p_\delta$ is bad.
\end{lem}
\begin{proof}
  Since there are no long splitting chains, there exists $i < \delta$ such that $p_\delta$ does not fork over $M_i$. By assumption, $p \rest M_i$ is bad, so by Remark \ref{bad-trans} $p_\delta$ is also bad, as desired.
\end{proof}

\begin{lem}\label{bad-succ-0}
  Assume that $\K$ is $\mu$-superstable. Let $M \in \K_\mu$ be a limit model. If $p \in \gS (M)$ is bad, then there exists a limit model $N$ in $\K_\mu$ with $M \lea N$ and $q_1, q_2 \in \gS (N)$ such that:

  \begin{enumerate}
  \item Both $q_1$ and $q_2$ extend $p$.
  \item $q_1 \neq q_2$.
  \item Both $q_1$ and $q_2$ are bad.
  \end{enumerate}
\end{lem}
\begin{proof}
  By definition, $p$ is $n$-bad for some $n < \omega$. We proceed by induction on $n$.

  \begin{itemize}
  \item If $n = 0$, this is the definition of being $0$-bad (note that $q_1$ and $q_2$ from Definition \ref{bad-def} are bad because they are nonforking extensions of the bad type $p$, see Remark \ref{bad-trans}) 
  \item If $n = m +1$, let $M_0 \in \K_\mu$ be a limit model such that $M_0 \lea M$, $p$ does not fork over $M_0$, and $p \rest M_0$ is $m$-bad. Pick $M_0'$ such that $p$ does not fork over $(M_0', M_0)$. Let $M_1'$ be $(\mu, \omega)$-limit over $M_0'$ with $M_1' \lea M_0$. By monotonicity, $p$ does not fork over $(M_0', M_1')$. Let $M^\ast$ be $(\mu, \omega)$-limit over $M$ (hence over $M_1'$). Let $q \in \gS (M^\ast)$ be an extension of $p$ that does not fork over $(M_0', M)$, hence over $(M_0', M_1')$. By Fact \ref{conj-prop}, $q$ and $p \rest M_1'$ are conjugate over $M_0'$. Now by the induction hypothesis, there exists a limit model $N^\ast$ extending $M_0$ and two distinct bad extensions of $p \rest M_0$ to $N^\ast$. These are also extensions of $p \rest M_1'$, so the result follows from the fact that $q$ and $p \rest M_1'$ are conjugate over $M_0'$.
  \end{itemize}
\end{proof}

The following nominally stronger version of Lemma \ref{bad-succ-0} (where $N$ is fixed first) is the one that we will use to show that there are no bad types:

\begin{lem}\label{bad-succ}
  Assume that $\K$ is $\mu$-superstable. Let $M$ be a limit model in $\K_\mu$ and let $N$ be limit over $M$. If $p \in \gS (M)$ is bad, then there exists $q_1, q_2 \in \gS (N)$ such that:

    \begin{enumerate}
  \item Both $q_1$ and $q_2$ extend $p$.
  \item $q_1 \neq q_2$.
  \item Both $q_1$ and $q_2$ are bad.
  \end{enumerate}
\end{lem}
\begin{proof}
  By Lemma \ref{bad-succ-0}, there exists $N' \in \K_\mu$ limit with $M \lea N'$ and $q_1', q_2' \in \gS (N')$ distinct bad extensions of $p$. Use universality of $N$ to pick $f: N' \xrightarrow[M]{} N$. For $\ell = 1,2$, let $q_\ell'' := f (q_\ell')$. Clearly, $q_1''$, $q_2''$ are still distinct bad extensions of $p$. Now for $\ell = 1,2$, let $q_\ell \in \gS (N)$ be an extension of $q_\ell''$ that does not fork over $f[N']$ (use no long splitting chains and extension). Then $q_1$ and $q_2$ are as desired (they are bad because they are nonforking extensions of the bad types $q_1'', q_2''$, see Remark \ref{bad-trans}).
\end{proof}

\begin{lem}\label{no-bad}
  If $\K$ is $\mu$-superstable, then there are no bad types.
\end{lem}
\begin{proof}
  Suppose for a contradiction that there is a limit model $M$ in $\K_\mu$ and a bad type $p \in \gS (M)$. Fix an increasing continuous chain $\seq{M_i : i \le \mu}$ with $M_0 = M$ and $M_{i + 1}$ limit over $M_i$ for all $i < \mu$. We build a tree of types $\seq{p_\eta : \eta \in \fct{\le \mu}{2}}$ satisfying:

  \begin{enumerate}
  \item $p_{<>} = p$.
    \item For all $\eta \in \fct{\le \mu}{2}$, $p_\eta \in \gS (M_{\ell (\eta)})$.
    \item For all $\nu \le \eta \in \fct{\le \mu}{2}$, $p_\eta$ is an extension of $p_\nu$.
    \item For all $\eta \in \fct{\le \mu}{2}$, $p_\eta$ is bad.
    \item For all $\eta \in \fct{<\mu}{2}$, $p_{\eta \smallfrown 0} \neq p_{\eta \smallfrown 1}$.
  \end{enumerate}

  \underline{This is enough}: The requirements give that for all $\eta, \nu \in \fct{\mu}{2}$, $\eta \neq \nu$ implies $p_\eta \neq p_\nu$. Therefore $|\gS (M_\mu)| = 2^\mu > \mu$, contradicting stability.

  \underline{This is possible}: We proceed by induction on $\ell (\eta)$. The base case has already been specified. At limits, we use Fact \ref{limit-fact} and Lemma \ref{bad-limit}. At successors, we use Lemma \ref{bad-succ}.
\end{proof}

\begin{thm}[Uniqueness of forking]\label{main-thm}
  Assume that $\K$ is $\mu$-superstable. Let $M \lea N$ both be limits in $\K_\mu$. Let $p, q \in \gS (N)$. If $p \rest M = q \rest M$ and both $p$ and $q$ do not fork over $M$, then $p = q$.
\end{thm}
\begin{proof}
  Otherwise, this would mean that $p \rest M$ is $0$-bad, contradicting Lemma \ref{no-bad}.
\end{proof}

\subsection{The hierarchy of symmetry properties collapses}

In \cite[\S4]{vv-symmetry-transfer-afml}, VanDieren and the author defined several variations of the symmetry property (we have highlighted the differences between each, see the previously-cited paper for more motivation):

\begin{defin}\label{many-syms} \
  \begin{enumerate}
  \item\label{many-syms-1} $\K$ has \emph{uniform $\mu$-symmetry} if for any limit models $N, M_0, M$ in $\K_\mu$ where $M$ is limit over $M_0$ and $M_0$ is limit over $N$, if \underline{$\gtp (b / M)$ does not $\mu$-split over $M_0$}, $a \in |M|$, and $\gtp (a / M_0)$  does not $\mu$-fork over $(N, M_0)$, there exists $M_b \in \K_{\mu}$ containing $b$ and limit over $M_0$ so that \textbf{$\gtp (a / M_b)$  does not $\mu$-fork over $\mathbf{(N, M_0)}$}.
    \item\label{many-syms-2} $\K$ has \emph{weak uniform $\mu$-symmetry} if for any limit models $N, M_0, M$ in $\K_\mu$ where $M$ is limit over $M_0$ and $M_0$ is limit over $N$, if \underline{$\gtp (b / M)$ does not $\mu$-fork over $M_0$}, $a \in |M|$, and $\gtp (a / M_0)$  does not $\mu$-fork over $(N, M_0)$, there exists $M_b \in \K_{\mu}$ containing $b$ and limit over $M_0$ so that \textbf{$\gtp (a / M_b)$ does not $\mu$-fork over $\mathbf{(N, M_0)}$}.
    \item\label{many-syms-3} $\K$ has \emph{non-uniform $\mu$-symmetry} if for any limit models $M_0, M$ in $\K_\mu$ where $M$ is limit over $M_0$, if \underline{$\gtp (b / M)$ does not $\mu$-split over $M_0$}, $a \in |M|$, and $\gtp (a / M_0)$ does not $\mu$-fork over $M_0$, there exists $M_b \in \K_{\mu}$ containing $b$ and limit over $M_0$ so that \textbf{$\gtp (a / M_b)$ does not $\mu$-fork over $\mathbf{M_0}$}.
    \item\label{many-syms-4} $\K$ has \emph{weak non-uniform $\mu$-symmetry} if for any limit models $M_0, M$ in $\K_\mu$ where $M$ is limit over $M_0$, if \underline{$\gtp (b / M)$ does not $\mu$-fork over $M_0$}, $a \in |M|$, and $\gtp (a / M_0)$ does not $\mu$-fork over $M_0$, there exists $M_b \in \K_{\mu}$ containing $b$ and limit over $M_0$ so that \textbf{$\gtp (a / M_b)$ does not $\mu$-fork over $\mathbf{M_0}$}.
  \end{enumerate}
\end{defin}

In \cite[\S4]{vv-symmetry-transfer-afml}, it was shown that the uniform variation corresponds to the symmetry property for splitting introduced by VanDieren in \cite{vandieren-symmetry-apal} and the weak non-uniform variation corresponds to the symmetry property of good frames (over limit models). It was also proven that $(\ref{many-syms-1}) \Leftrightarrow (\ref{many-syms-2}) \Rightarrow (\ref{many-syms-3}) \Rightarrow (\ref{many-syms-4})$. Using Theorem \ref{main-thm}, it is now easy to show that all these properties are equivalent.

\begin{cor}\label{sym-equiv-cor}
  If $\K$ is $\mu$-superstable, then uniform $\mu$-symmetry is equivalent to weak non-uniform $\mu$-symmetry.
\end{cor}
\begin{proof}
  We show that weak non-uniform $\mu$-symmetry implies weak uniform $\mu$-symmetry, which is known to be equivalent to uniform $\mu$-symmetry \cite[4.6]{vv-symmetry-transfer-afml}. So assume that we are given $N, M_0, M, a, b$ as in the definition of weak uniform $\mu$-symmetry. Let $M_b$ be as given by the definition of weak \emph{non-uniform} $\mu$-symmetry. We know that $\gtp (a / M_b)$ does not $\mu$-fork over $M_0$, but we really want to conclude that it does not $\mu$-fork over $(N, M_0)$.

  By assumption, $\gtp (a / M_0)$ does not $\mu$-fork over $(N, M_0)$. Therefore by extension there is $a'$ such that $\gtp (a/ M_0) = \gtp (a' / M_0)$ and $\gtp (a' / M_b)$ does not $\mu$-fork over $(N, M_0)$. We have that $\gtp (a / M_b)$, $\gtp (a' / M_b)$ both do not $\mu$-fork over $M_0$ therefore by uniqueness (Theorem \ref{main-thm}), $\gtp (a / M_b) = \gtp (a' / M_b)$. In particular, $\gtp (a / M_b)$ does not $\mu$-fork over $(N, M_0)$, as desired.
\end{proof}

\bibliographystyle{amsalpha}
\bibliography{uq-forking}

\end{document}